\documentclass{article}
\usepackage{graphicx} 

\usepackage{amsthm,amsfonts,amssymb,amsmath,epsf, verbatim}
\usepackage{hyperref}

\newtheorem{theorem}{Theorem}
\newtheorem{lemma}[theorem]{Lemma}

\newtheorem{proposition}[theorem]{Proposition}

\newtheorem{conjecture}[theorem]{Conjecture}
\newtheorem{question}[theorem]{Question}

\title{On near superperfect numbers, the Goormaghtigh conjecture, and Mertens' theorem} 
\author{Satvik Beri\footnote{satvik.beri@gmail.com} , Joshua Zelinsky\footnote{Department of Mathematics, Hopkins School, New Haven, CT. USA \\ \indent jzelinsky@hopkins.edu}}
\date{}

\begin{document}

\maketitle

\begin{abstract} Let $\sigma(n)$ be the sum of the divisors of $n$. Kalita and Saikia defined a number $n$ to be near superperfect if $2n+d=\sigma(\sigma(n))$ for some positive divisor $d$ of $n$.  We extend some of their results about near superperfect numbers and connect these results to the Goormaghtigh conjecture and to certain products of primes similar to those which appear in Mertens' theorem. We also define type II near superperfect numbers, which are those $n$ satisfying $2n+d=\sigma(\sigma(n))$ for some positive divisor $d$ of $\sigma(n)$, and prove analogous results about these numbers.
\end{abstract}

\section{Introduction}

Let $\sigma(n)$ be the sum of the positive divisors of $n$. A number is said to be \textit{perfect} if $n$ satisfies the equation $\sigma(n) =2n$. A number is said to be \textit{abundant} if $\sigma(n) > 2n$ and is said to be \textit{deficient} if $\sigma(n) < 2n$. Two of the oldest unsolved problems in mathematics are whether there are infinitely many even perfect numbers and whether there are any odd perfect numbers. A wide variety of work on the second problem has occurred in the last few years. 

For example, Acquaah and Konyagin \cite{AK} showed that if $n$ is an odd perfect number, then its largest prime factor is at most $(3n)^{\frac{1}{3}}$. Two subsequent papers, one by the second author and the other by the second author with Bibby and Vyncke \cite{JZ second, BVZ},  proved similar upper bounds for the second and third largest prime factor.  In terms of lower bounds, Goto and Ohno \cite{Goto and Ohno} proved that the largest prime factor of such a number must be at least $10^8$. Similar lower bounds for the second and third largest prime factor are due to Iannucci \cite{Iannucci second, Iannucci third}.

There are also similar bounds for the size of such an $n$ in general, with Ochem and Rao \cite{Ochem Rao} showing that one must have $n< 10^{1500}$. Ochem in as-yet unpublished work extended this lower bound to $10^{2000}$ but pushing this bound further will be difficult without substantial new insights. Upper bounds for $n$ in terms of its number of prime factors have also been proven. The first such bound, which grows rapidly, was given by Pomerance \cite{Pomerance}. This was subsequently greatly reduced by Heath-Brown \cite{Heath-Brown} who showed that if $n$ is an odd perfect number with $k$ distinct prime factors then we must have

$$n < 4^{4^k}.$$

Subsequent papers by Cook \cite{Cook}, Nielsen \cite{Nielsen upper}, and finally by Chen and Tang \cite{Chen and Tang} have reduced this to $$n < 2^{4^k -2^k}.$$

A recently investigated bound is of the following type. Assume that $n$ is an odd perfect number where $n= p_1^{a_1}p_2^{a_2}\cdots p_k^{a_k}$ where the $p_i$ are distinct primes. Then let $\Omega(n)= a_1 + a_2 \cdots +a_k$. That is, $\Omega(n)$ is the total number of prime factors of $n$, counting repetitions. 

 Ochem and Rao \cite{OchemRao1}  have proved that $N$ must satisfy \begin{equation}\Omega(N) \geq \frac{18k -31}{7}\label{OR1}\end{equation} and \begin{equation}\Omega(N) \geq 2k +51.\label{OR2} \end{equation}

This bound was subsequently improved by the second author in a pair of papers \cite{Zelinsky1, Zelinskybig} which were then surpassed by a paper by Clayton and Hansen \cite{Clayton and Hansen} who proved that $$\Omega(n) \geq \frac{99}{37}k -\frac{187}{37},$$ and that if $3 \nmid n$, then one has the stronger result that

$$\Omega(n) \geq \frac{51}{19}k - \frac{46}{19}. $$
 
A recent direction has been to look at what Dris and Luca \cite{Dris and Luca} called the {\emph{index}} of a prime power of $n$. In particular, assume that $p^a||n$ (that is, that $p^a|n$ and $p^{a+1} \nmid n$). Define the index of $p^a$ with respect to $n$ to be $\sigma(n/p^a)/p^a$. They showed that for any such $n$ and $p^a$, the index is greater than 5. This was greatly improved by Chen and Chen in a pair of papers \cite{Chen and Chen1, Chen and Chen2} who showed among other results that the index is never of any of the forms $q_1^i$ where $q_1$ is prime and $1 i \leq 4$. They also similarly ruled out the index being $q_1q_2$ or $q_1^2q_2$ for two distinct primes, along with 24 similar forms. This was then improved further by Gallardo \cite{Gallardo} who showed how to rule out many additional forms.  

Despite these restrictions and many other results, we appear to be not anywhere close to proving the non-existence of odd perfect numbers. Not surprisingly then, many have tried to branch out to look at closely related problems, and a large number of variations on the idea of perfect numbers have been proposed.

Pollack and Shevelev \cite{PS} defined a number $n$ to be \textit{near perfect} if it satisfies  $2n+d=\sigma(n))$ for some $d >0$ where $d|n$. The $d$ in question is referred to as the omitted divisor or redundant divisor since one has that $2n=\sigma(n)$ if one omits $d$ from the sum of the divisors of $n$.  A variety of results concerning near perfect numbers have recently been proven.  In their original paper, Pollack and Shevelev constructed three families of near perfect numbers: 

\begin{enumerate}
\item $2^{t-1}(2^t-2^k-1)$ where $2^t-2^k-1$ is prime. Here $2^k$ is the omitted divisor.
\item $2^{2p-1}(2^p-1)$ where $2^p-1$ is prime. Here $2^p(2^p-1)$ is the omitted divisor.
\item $2^{p-1}(2^p-1)^2$ where $2^p-1$ is prime. Here $2^p-1$ is the omitted divisor.
\end{enumerate}

Subsequent work by Ren and Chen \cite{RC} showed that all near perfects with two distinct prime factors  must be either 40 or a member of one of Pollack and Shevelev's three families. 

Hasanalizade \cite{Hasanalizade} proved that there are no odd near   perfect Fibonacci or Lucas numbers, and that there are no near perfect Fibonacci numbers with fewer than three distinct prime factors. Tang, Ma, and Feng \cite{TMF} showed that  $173369889=(3^4)(7^2)(11^2)(19^2)$  is the only odd near perfect number with four or fewer distinct prime divisors; currently this is the only odd near perfect number known. 

Similarly, Pollack and Shevelev defined a number to be $s$-\textit{near perfect} if $$\sigma(n) =2n + d_1 + d_2 \cdots d_s$$ where $d_1, d_2, \cdots d_s$ are distinct positive divisors of $n$. Note that when $s=1$ these are just the usual near perfect numbers. Aryan, Madhavani, Parikh, Slattery, and the second author of this paper, in a recent paper \cite{AMPSZ} classified all $2$-near perfect numbers of the form $2^kp$ or $2^k p^2$ where $p$ is an odd prime. In particular, they proved the following two theorems.

\begin{theorem} Assume $n$ is a $2$-near perfect number with omitted divisors $d_1$ and $d_2$. Assume further that $n=2^k p$ where $p$ is prime and $k$ is a positive integer. Then one must have, without loss of generality, one of four situations. \label{classification of 2 near perfect of form power of two times a prime}
\begin{enumerate}
    \item $p=2^k-1$. Here we have $d_1=1$ and $d_2=p$.
    \item $p=2^{k+1} -2^a -2^b-1$ for some $a, b \in \mathbb{N}$. Here  $d_1=2^a$ and $d_2=2^b$.
    \item $p=\frac{2^{k+1}-2^a-1}{1+2^b}$ for some $a, b \in \mathbb{N}$. Here  $d_1 = 2^a$ and $d_2 = 2^bp$. 
    \item $p=\frac{2^{k+1}-1}{1+2^a+2^b}$ for some $a, b \in \mathbb{N}$. Here $d_1 =2^ap$ and $d_2=2^bp$.
\end{enumerate}
\end{theorem}

\begin{theorem} Assume that $n$ is a $2$-near perfect number with omitted divisors $d_1$ and $d_2$. Assume further that $n=2^kp^2$ where $p$ is prime. Then $n \in \{18, 36,200\}$. \label{Second main result power of 2 times square of a prime}
\end{theorem}
 Cohen, Cordwell, Epstein, Kwan, Lott, and Miller \cite{Miller group} gave strict bounds on the number of $s$-near perfect numbers under a given $x$ as long as $s$ is fixed and at least four. Currently, the behavior of $3$-near perfect numbers seems to be poorly understood. When $k=1$ and $k=2$, the small number of divisors gives little flexibility with what $k$-near perfect numbers can look like, and for $s \geq 4$ there is enough flexibility to say without any conditions that there are many $s$-near perfect numbers. However, $s=3$ seems to be a middle ground of difficulty, having neither the simplicity of small $s$ nor the high flexibility of large $s$.

A thematically related idea was investigated by Davis, Klyve, and Kraght \cite{Davis Klyve Kraght} who studied the image of the function $e(n) = \sigma(n)-2n $, which they termed the \textit{excedent} function, although another common term for this quantity is the abundance of $n$.\footnote{The etymology behind this term is notable: Their term ``excedent'' is based on the Italian word  \textit{eccedenza} used by Cattaneo who seems to have been the first person to introduce the function in \cite{Cattaneo}.} Among other results, they showed that every Mersenne prime is in the image of $e(n)$, using the same construction of Pollack and Shevelev for their third family of near perfect numbers. They also showed that aside from a set of density zero (which may in fact be the empty set) if $n$ is of the form $n \equiv 12$ (mod 24), then $n$ is in the image of $e(n).$  Davis, Klyve and Kraght conjectured that every even integer is in the range of $e(n)$.  Related work  was then done by a paper by Dean, Erdman, Klyve, Lycette,  Pidde,  and Wheel \cite{DKELPW}. Note that the abundance $e(n)$ should not be confused with the abundancy index or just abundancy $h(n)= \frac{\sigma(n)}{n}$ which is in many respects a better behaved function and is the jumping off point for much research on odd perfect numbers. Most of the advantages of $h(n)$ over $e(n)$ stem from $h(n)$ being a multiplicative function. 
 
In a different direction, Suryanarayana \cite{Suryanarayana} defined a number $n$ to be \textit{superperfect} if $n$ satisfies $2n = \sigma(\sigma(n))$. An even number $n$ is superperfect exactly when $n=2^k$ and $2^{k+1}-1$ is a Mersenne prime. It is an open problem if there are any odd superperfect numbers. Kanold \cite{Kanold} proved that any odd superperfect number must be a perfect square. Yamada \cite{Yamada superperfect} proved that there are only finitely many odd superperfect numbers with $k$ distinct prime divisors. Suryanarayana \cite{Suryanarayana 1973} proved that there are no odd superperfect numbers which are a power of a prime.

In a recent paper \cite{Kalita and Saikia}, Kalita and Saikia hybridized the definition of near perfect with the definition of superperfect. They defined a number $n$ to be \textit{near superperfect} if $2n+d=\sigma(\sigma(n))$ for some positive divisor $d$ of $n$. The primary focus of this paper will be extending the results of their paper as well as suggesting additional hybrid definitions.   

Kalita and Saikia proved that $n$ is near superperfect when $n$ is a Mersenne prime. They also found two other near superperfect numbers: $n=8$ and $n=512$. They asked if all near superperfect numbers are Mersenne primes or powers of 2, and also asked if 8 and 512 were the only powers of 2 which are near superperfect. Kalita and Saikia also proved that there is no near superperfect number of the form $n=p^a$ for prime $p$ when $\sigma(p^a)=p^a + p^{a-1} \cdots +1$ is prime. Our results will focus on extending their results about odd near superperfect numbers as well as some related results.

We will say a number $n$ is a Kalita-Saikia number if it satisfies that $\sigma(p^a)$ is prime whenever $p^a||n$. (Here $p^a||n$ means $p^a|n$ and $p^{a+1}\nmid n$.) Kalita-Saikia numbers which are near superperfect  turn out to be a particular nice family of numbers to work with, due to   $\sigma(\sigma(n)$ being much easier to characterize than it is for other numbers. The first few Kalita-Saikia numbers are :

$$1,2, 4, 9, 16, 18, 25, 36, 50, 64, 100, 144, 225, 289, 400, 450, 576, 578, 729, 900, 1156.$$

The sequence of Kalita-Saikia numbers is not currently in the OEIS.\\

Our main result is the following generalization of their result.

\begin{theorem} \label{strong version of KS1} Suppose that $n$ is odd and $n=p_1^{a_1} \cdots p_k^{a_k}$ for distinct primes $p_1 < \cdots < p_k$. Suppose further that for any $i$ where $1 \leq i \leq k$, 
$\sigma(p_i^{a_i})$ is prime. Finally, assume that $n$ is a near superperfect number. 

Given the above assumptions, we have $k \geq 3$. Moreover, if $k=3$, $k+3$ or $k=5$, then we must have $p_1=3$. 
\end{theorem}

The second section of this paper proves this theorem as well as related results, including connections to the  the Goormaghtigh conjecture and Mertens' theorem.

The third section of this paper discusses improved algorithms for searching for near superperfect numbers. The fourth section discusses variants of the idea of near superperfect numbers.

\section{Extending Kalita and Saikia's results for odd near superperfect numbers}

To prove Theorem \ref{strong version of KS1}, we will prove the following Lemma.
\begin{lemma} Assume that $n$ is an odd positive integer where $n=p_1^{a_1} \cdots p_k^{a_k}$ and $p_1 < \cdots p_k$ are prime. Assume further  that for any $i$ where $1 \leq i \leq k$, 
$\sigma(p_i^{a_i})$ is prime. Then 
$$\frac{\sigma(\sigma(n))}{n} \leq \left(\frac{p_1}{p_1-1} + \frac{1}{p_1^2}\right) \cdots \left(\frac{p_k}{p_k-1} +\frac{1}{p_k^2} \right).$$ \label{Modified abundance lemma}    
\end{lemma}
\begin{proof} Assume that $n$ is an odd positive integer satisfying the assumptions in the Lemma. Note that in order for $\sigma(p_i^{a_i})$ to be prime when $p_i$ is odd, $\sigma(p_i^{a_i})$ must be odd, and therefore $a_i$ is even. Thus $a_i \geq 2$ for all $i$. We now note that

\begin{equation}\begin{split}\frac{\sigma(\sigma(p_i^{a_i}))}{p_i^{a_i}} & = \frac{\sigma(\frac{p_i^{a_i+1}-1}{p_i-1})}{p_i^{a_i}} \\ &= \frac{\frac{p_i^{a_i+1}-1}{p_i-1}+1}{p_i^{a_i}} = \frac{p_i}{p_i-1} + \frac{1}{p_i^2}  -\left(\frac{1}{p_i^{a_i}(p_i-1)} - \frac{1}{p_i^{a_i}} + \frac{1}{p_i^2} \right).\end{split}\end{equation}
Since $a_i$ is at least 2, we have that $$\frac{1}{p_i^{a_i}(p_i-1)} - \frac{1}{p_i^{a_i}} + \frac{1}{p^2} >0.$$

Thus, for any $i$, we have $$\frac{\sigma(\sigma(p_i^{a_i}))}{p_i^{a_i}} < \frac{p_i}{p_i-1} + \frac{1}{p_i^2}.$$ We then have $$\frac{\sigma(\sigma(n))}{n} \leq \frac{\sigma(\sigma(p_1^{a_1}))}{p_1^{a_1}}  \cdots \frac{\sigma(\sigma(p_k^{a_k}))}{p_k^{a_k}} <  \left(\frac{p_1}{p_1-1} + \frac{1}{p_1^2}\right) \cdots \left(\frac{p_k}{p_k-1} +\frac{1}{p_k^2} \right). $$ 
Note that the  first inequality above is a less than or equal rather than an equality in the above is not an equality because we may have some $i$ and $j$ where $i \neq j$ but $\sigma(p_i^{a_i}) = \sigma(p_j^{a_j})$, and this may make $\sigma(\sigma(n))$ smaller than if no such collision occurred. 
\end{proof}

We now prove Theorem \ref{strong version of KS1}.

\begin{proof} First, consider the case where we have exactly two distinct prime divisors. Assume that $n$ is a near superperfect number where $n=p^aq^b$ for some distinct pair of primes $p$ and $q$ with $p < q$. Assume further that $\sigma(p^a)$ and $\sigma(q^b)$ are both prime. Thus,
we have some $d$ where $d|n$, and

\begin{equation} 
    2n +d = \sigma(\sigma(n))
\end{equation}

If $p \neq 3 $, then by Lemma \ref{Modified abundance lemma}, we have 

\begin{equation}
    2 < \left(\frac{p}{p-1} + \frac{1}{p^2}\right)\left(\frac{q}{q-1} + \frac{1}{q^2}\right) \leq \left(\frac{5}{4} + \frac{1}{25}\right) \left(\frac{7}{6}+ \frac{1}{49}\right) = \frac{15007}{9800} < 2. \label{5 7 inequality } 
\end{equation}

Thus, we must have $p=3$. By similar logic, if $q \geq 7$, then we have

\begin{equation}
    2 < \left(\frac{p}{p-1} + \frac{1}{p^2}\right)\left(\frac{q}{q-1} + \frac{1}{q^2}\right) \leq \left(\frac{3}{2} + \frac{1}{9}\right) \left(\frac{7}{6}+ \frac{1}{49}\right) = \frac{5249}{2700} < 2. \label{5 7 inequality 2 } 
\end{equation}

Thus, we must have $q=5$. 

Thus, our number must be of the form $n=3^a 5^b$. Since $3^a + 3^{a-1} \cdots +1$ is prime, $a$ is even. Similarly, $b$ is also even. Assume temporarily that $a=2$. We can verify that $n=3^2 5^2$ and $n=3^2 5^4$ does not have the desired form. Thus, we must have $b \geq 6$. 

Since $b \geq 6$, we have  

\begin{equation} 2(3^a)(5^b) = 2n < 2n +d = \sigma(\sigma(n)) = \left(\frac{3^{a+1}-1}{2} +1\right)\left(\frac{5^{b+1}-1}4 +1\right).  \label{explicit equation for just 3 and 5}
\end{equation}

Equation \ref{explicit equation for just 3 and 5} implies that

\begin{equation}
    2 <  \left(\frac{\frac{3^{a+1}-1}{2} +1}{3^a}\right)\left(\frac{\frac{5^{b+1}-1}{4} +1}{5^b}\right) = \left(\frac{3}{2} + \frac{1}{2 (3^a)}\right)\left(\frac{5}{4} + \frac{3}{4 (5^b)}\right).\label{penultimate 3 5 equation}
\end{equation}

However, we have \begin{equation*}
    \left(\frac{3}{2} + \frac{1}{2 (3^a)}\right)\left(\frac{5}{4} + \frac{3}{4 (5^b)}\right) \leq \left(\frac{3}{2} +\frac{1}{18}\right)\left(\frac{5}{4} + \frac{3}{4(5^6)}\right) =  1.94 < 2, 
\end{equation*}
which contradicts Equation \ref{penultimate 3 5 equation}.

Thus, we must have $a >2$. Note that $\sigma(3^4)$ is not prime, and so we must have $a \geq 6$. But this leads to a contradiction of the same sort as in the previous situation. 

Thus, we must have $k \geq 3$. To see that that if $k \leq 5$, we must have $p_1=3$ note that if $p_1 \geq 3$ then
\begin{equation}\begin{split}2 < \frac{\sigma(\sigma(n))}{n} & < \left(\frac{5}{4}+\frac{1}{25}\right)\left(\frac{7}{6}+\frac{1}{49}\right)\left(\frac{11}{10}+\frac{1}{121}\right)\left(\frac{13}{12}+\frac{1}{169}\right)\left(\frac{17}{16}+\frac{1}{289}\right) \\ & <2,\end{split}\end{equation} which is a contradiction.
\end{proof}

A brief digression concerning Lemma \ref{Modified abundance lemma}: This Lemma and its use is very similar to a classic result that for any positive integer $n$, 
\begin{equation} \frac{\sigma(n)}{n} \leq \prod_{p|n} \frac{p}{p-1} \label{classic upper bound on sigma}
\end{equation}
with equality only when $n=1$.

Motivated by trying to understand odd perfect numbers, Norton defined the following function:
Let $P_n$ be the $n$th prime number. Norton defined $a(n)$, as the least such value such that \begin{equation} 2<  \prod_{r=n}^{n+a(n)-1} \frac{P_r}{P_r-1}. \label{Norton-Grun definition}
\end{equation}

It is easy to see that if $N$ is an odd perfect number or an odd abundant number with smallest prime divisor $P_n$, then $N$ must have at least $a(n)$ distinct prime divisors. Norton used the fact that  a similar product appears in Mertens' theorem to estimate $a(n)$. Thus, study of $a(n)$ is a natural object even if one is not strongly interested in odd perfect numbers. Subsequent work by the second author \cite{Zelinskybig} tightened those estimates and asked a variety of questions about the growth of $a(n)$. The second author discussed certain generalizations of this function in \cite{Zelinskyfollowup}.

Define the Kalita-Sakia function $KS(n)$ to be the least value such that \begin{equation} 2<  \prod_{r=n}^{n+a(n)-1} \frac{P_r}{P_r-1} + \frac{1}{P_r^2}. \label{Kalita and Saikia variant of Norton's function}
\end{equation}

Given the proof above of Theorem \ref{strong version of KS1},  $KS(n)$ seems like a natural function occurring in connection to near superperfects in a way similar to how Norton's function appears in the study of perfect numbers.  Note that $a(2)=3$, but $KS(2)=2$. It is also the case that $a(3) > KS(3)=5$.  This leads to two questions.

\begin{question} Is $a(n) > KS(n)$ for all $n$?
\end{question}

\begin{question} Is $ \displaystyle\lim_{n \rightarrow \infty} a(n)-KS(n) = \infty?$ \label{KS second question}    
\end{question}

However, while these questions are interesting, we can extend our results about odd near superperfect numbers using a different line of reasoning. Notice that if $n$ is odd, then $2n+d$ is odd for any $n$. But  in order for $\sigma(m)$ to be odd one must have $m=2^k b^2$ for some $k$ and $b$. Thus, in particular, in order for an odd $n$ to be near superperfect, one must have $\sigma(\sigma(n))$ odd, and thus $\sigma(n)=2^k b^2 $ for some $k$ and $b$. Thus, if $n$ is an odd near superperfect number and $n$ is a Kalita-Saikia number,   then $\sigma(n)$ must be a perfect square.

However, this condition requires $n$ to have many collisions of primes. That is, for any prime $p$ where $p^a||n$, there must be a prime $q$ and an integer $b$ such that $q^b||n$ and $\sigma(p^a)= \sigma(q^b)$. Such collisions are extremely rare. The  Goormaghtigh conjecture states this in a very broad way. 

\begin{conjecture}(Goormaghtigh) The only positive integers $x$, $y$, $m$ and $n$ where $x > y$ and $m$, $n, \geq 2$, and satisfying $$\frac{x^m - 1}{x-1} = \frac{y^n - 1}{y-1}$$

are $(x,y,m,n)= (5,2,3,5)$ and $(90,2,3,13)$.
\end{conjecture}

In order for $n$ to be an Kalita-Saikia number and also be a near superperfect number, $n$ needs the following property: For any prime $p$ where $p^a||n$, we must have $\sigma(p^a)$ prime, and we must have a prime $q$ and a $b$ such that the quadruplet forms a counterexample to the Goormaghtigh conjecture. In this context it seems that a safe (perhaps overly safe) conjecture is the following.

\begin{conjecture} There are no odd Kalita-Saikia near superperfect numbers. \label{conjecture that there no odd Kalita-Saikia near superperfect}   
\end{conjecture}

By the above remarks, the Goormaghtigh conjecture implies Conjecture \ref{conjecture that there no odd Kalita-Saikia near superperfect}.

In fact, we can use some of the known results of the Goormaghtigh conjecture to tighten some of our above results. In particular, Yuan \cite{Yuan} proved that the Goormaghtigh conjecture is true for $n=3$. This means that for any odd prime $p$, $\sigma(p^2)$ can never have a collision with $\sigma(q^b)$ for an odd prime $q$. 

We thus have the following theorem.
\begin{theorem} If $n$ is an odd Kalita-Saikia number which is near superperfect, then every prime factor of $n$ must be raised to at least the fourth power. \label{KS numbers near superperfect must have all prime factors raised to at least the fourth power}
\end{theorem}

Theorem \ref{KS numbers near superperfect must have all prime factors raised to at least the fourth power} motivates the following variant of Lemma \ref{Modified abundance lemma} whose proof uses the same method:

\begin{lemma} Assume $n$ is a Kalita-Saikia number, and suppose that for any prime $p$ where $p|n$ we have $p^4|n$.  Then $$\frac{\sigma(\sigma(n)}{n} \leq \prod_{p|n} \frac{p}{p-1} + \frac{1}{p^4}.$$   \label{upper bound for sigma sigma for KS numbers with 4th power}
\end{lemma}

An immediate consequence of Lemma \ref{upper bound for sigma sigma for KS numbers with 4th power} is the following.

\begin{proposition} Suppose $n$ is an odd  Kalita-Saikia number which is near superperfect. If $n$ is not divisible by 3, then $n$ must have at least 7 distinct prime factors.
\end{proposition}

Similarly to how $a(n)$ and $KS(n)$ are defined we may now define $KS_4(n)$ as the least number such that 

\begin{equation} 2<  \prod_{r=n}^{n+KS_4-1} \frac{P_r}{P_r-1} + \frac{1}{P_r^4}. \label{Kalita and Saikia second variant of Norton's function}
\end{equation}

Note that $KS_4(2)=a(2)$ and $KS_4(3)=a(3)$. 

\begin{question} Is $KS_4(n)=a(n)$ for all $n \geq 2$?
\end{question}

More generally, let us define the Kalita-Norton-Saikia function as follows. $KNS_{\alpha,c}(n)$ is the least number such that 

\begin{equation} 2<  \prod_{r=n}^{n+KNS_\alpha-1} \frac{P_r}{P_r-1} + \frac{c}{P_r^\alpha}. \label{Kalita and Saikia generalform of Norton's function}
\end{equation}

Note that $KNS_{1,0}(n) =a(n) $, $KNS_{2,1}(n) = KS(n)$. Similarly,  $KNS_{4,1} = KS_4(n)$.

Two questions naturally arise. Let $U$ be the set of $\alpha$ such that there exists a $c >0$ such that $KNS(\alpha,c)(n) > a(n)$ for all sufficiently large $n$. Let $V$ be the set of $\alpha$ such that there exists a $c>0$ such that $KNS(\alpha,c)(n) > a(n)$ for infinitely many $n$.

\begin{question} What is the supremum of $U$?
\end{question}

\begin{question} What is the supremum of $V$?
\end{question}

There are also substantial restrictions on how similar the factorization of numbers of the form $\frac{x^\ell-1}{x-1}$ can be for different choices of $x$ and for a fixed value $\ell$. See for example, \cite{Yamada}. It seems plausible that this type of result can be used to substantially further restrict what an odd near superperfect number can look like or to rule out the case of a Kalita-Saikia odd near superperfect.

It is well known that the set of perfect numbers has natural density zero, and it is not hard to show that the same is true for the set of superpefect numbers. One might wonder if this was also the case for the near superperfect, since the presence of a ''+d'' in the definition gives some flexibility. However,  they still have natural density zero. To prove this, we need a Lemma.

\begin{lemma} Let $C$ be a positive constant. Then the set of numbers $n$ satisfying $\sigma(\sigma(n)) > C\sigma(n)$ has density 1. \label{Arb large h(sigma(n)) on set of density 1}
\end{lemma}

Lemma \ref{Arb large h(sigma(n)) on set of density 1} follows immediately from following Lemma.

\begin{lemma} For any integer $m$, the set of $n$ such that $m!|\sigma(n)$ has density 1.
\end{lemma}
\begin{proof} Fix a positive integer $m$. The strong form of Dirichlet's theorem on arithmetic progressions implies that the set of $n$ with at least one prime factor $p$ such that $p||n$ and $p \equiv -1 \pmod{m!}$ has density 1. Since for any such number $n$ which has a prime divisor of that form has $m!|\sigma(n)$, the result follows. 
\end{proof}

We now ready to make a concrete statement about the density of near superperfect numbers. 

\begin{proposition} Near superperfect numbers have natural density zero. \label{density zero}  
\end{proposition}
\begin{proof} A number $n$ cannot be near superperfect if $\sigma(\sigma(n)) > 3n$, since $2n+d \leq 3n$ no matter the choice of $d$. However, the set of $n$ such that $\sigma(\sigma(n)) > 3n$ is a set of density 1 by Lemma \ref{Arb large h(sigma(n)) on set of density 1}, since any number where $5!|\sigma(n)$ will satisfy this bound.
\end{proof}

We also here list three additional questions about Kalita-Saikia numbers.

\begin{question} $(2,4)$, $(16,18)$, $(576,578)$ are examples of pairs of Kalita-Saikia numbers which are exactly 2 apart. Are these the only examples?
\end{question}

\begin{question} Is there any $n$ such that $n$ and $n+1$ are both Kalita-Saikia numbers?\label{consecutive KS numbers}
\end{question}

We suspect that the answer to Question \ref{consecutive KS numbers} is no. If $n$ and $n+1$ are consecutive Kalita-Saikia numbers, then one of them must be a perfect square and the other must be two times a perfect square. These thus would be arising as a special case of the solutions to two Pell equations. While such equations have have infinitely many solutions, the solutions themselves grow exponentially. At a minimum then, it seems safe at least to conjecture that there are only finitely many $n$ such that $n$ and $n+1$ are both Kalita-Saikia numbers. \\

The Kalita-Saikia numbers empirically appear to be a subsequence of OEIS sequence A337343. Here, A337343 is defined as follows. Define $t(n)$ as a multiplicative function, satisfying $t(2^m)=1$ for any power of 2, and for any odd prime $p$, $t(p^e)=q^e$ where $q$ is the largest prime less than $p$. Then a number $n$ is in A337343 when $t(\sigma(n)) \geq n$.
    
\begin{question} Does every Kalita-Saikia number $n$ satisfy 
$t(\sigma(n)) \geq n$?
\end{question}

\section{Searching for near superperfect numbers}

At the end of their paper, Kalita and Saikia asked two questions:

 \begin{enumerate}
     \item Are all near superperfect numbers either powers of 2 or Mersenne primes?
     \item Are 8 and 512 the only powers of 2 which are near superperfect?
 \end{enumerate}

The first question has an answer of ``no'' since $21$ is near superperfect. However, this seems to be the only exception. We therefore propose the  modified question: ``Are all near superperfect numbers either 21, powers of 2, or Mersenne primes?''

This repaired question motivates a more systematic search for near superperfect numbers. The naive algorithm is to calculate $\sigma(\sigma(n))$, and then see when $\sigma(\sigma(n))-2n|n$. However, it turns out that in many situations,  the full calculation is not required. In particular, information about $n$ and $\sigma(n)$ allows one often to conclude that $n$ is not near superperfect without going through the extra effort of calculating $\sigma(\sigma(n))$.

A search for such near superperfect numbers will also naturally detect the situation where one has $2n-d = \sigma(\sigma(n))$, for some positive divisor of $d$. A number $n$ such that $2n-d= \sigma(n)$ is referred to in the literature as a \textit{deficient-perfect number}. Deficient perfect numbers have has been extensively studied. See for example work by Tang and Feng \cite{Tang Feng}, as well as work by Tang, Ren, and Li \cite{Tang Ren Li} and and that of Yang and Tobe\cite{Yang Tobe}. By analogy, we will say a number $n$ is \textit{deficient superperfect} if it satisfies $2n-d = \sigma(\sigma(n))$ for some positive divisor $d$ of $n$. At present, the only known such number is $n=1$. 

Let $Q$ be the set of numbers which are either near superperfect or are deficient perfect. 

\begin{proposition} If $n \in Q$  and $n$ is odd, then $\sigma(n)= 2^km^2$ for some odd $m$.
\end{proposition}
\begin{proof} Assume that $n$ is odd and superperfect. Thus, we have there is some $d|n$ where
$$2n + d = \sigma(\sigma(n)),$$ Since $n$ is odd, so is $d$, and hence $2n+d$ is odd. Thus, $\sigma(\sigma(n))$ is also odd. Note that for any $a$ $\sigma(a)$ is odd if and only if $a$ is a power of 2 times an odd square, and so in this circumstance $\sigma(n)$ must be a power of 2 times an odd square.
\end{proof}

The above proposition means that for the vast majority of odd $n$, when $\sigma(n)$ is calculated, there is no need to calculate $\sigma(\sigma(n))$ since in the vast majority of circumstances $\sigma(n)$ is not of the needed form.

From the same sort of logic as the proof of Proposition \ref{density zero}, we have the following. 
\begin{lemma} If $n \in Q$, then $\sigma(n) \leq 3n$.  \label{sigma(n) at most 3n when n near super}
\end{lemma}

The next two Lemmas are easy corollaries of this result. 

\begin{lemma} If $2|\sigma(n)$ and $\sigma(n) > 2n$, then $n \not \in Q$.
\end{lemma}
\begin{proof} Assume that $2|\sigma(n)$ and $\sigma(n) > 2n$ We have then $\sigma(\sigma(n)) > \frac{3}{2}\sigma(n) \geq \frac{3}{2}(2n) = 3n$.
\end{proof}

\begin{lemma} If $6|n$ and $n \in Q$ then $n=2^m a$ where $a$ is an odd square. 
\end{lemma}

We also have through similar logic:

\begin{lemma} Let $8||n$ and $n > 8$ then $n \not \in Q$. \label{no 8s lemma} 
\end{lemma}
\begin{proof} If $8||n$, then $\sigma(8)=15|\sigma(n)$ and so one can check that $\sigma(\sigma(n)) \geq 3n$ with  with equality if and only $n=8.$
\end{proof}

In this case, $8$ is an exception because $8$ is near superperfect with $d=8$,  the largest possible option of the divisors of 8. Curiously, the same thing happens for $n=512$, where the divisor needed to make $512$ superperfect is also $d=n$. However, divisibility by $512$ is rare enough that checking for $2^9||n$ is not likely to be a substantial time saver for algorithmic purposes. 

We also have the following through similar logic. 
\begin{lemma} If $3||n$ and $p||n$ for some odd prime $p$ where $p \equiv 2$ (mod 3), then $n \not \in Q $.     
\end{lemma}

Note that we may slightly tighten Lemma \ref{sigma(n) at most 3n when n near super} although the improvement is not enough to make a major practical difference from an algorithmic standpoint.

\begin{lemma} If $n$ is near superperfect, then $\sigma(n) \leq 3n - \frac{\sqrt{12n+1}-1}{2}$.
\end{lemma}
\begin{proof} 
First, assume $n$ is even. We may assume that $n$ is not a power of 2, since if $n=2^k$,
then $$\sigma(n)= 2n-1 < 3n - \frac{\sqrt{12n+1}-1}{2}.$$  Therefore, $n$ must have at least two distinct prime factors, and so $\sigma(n)$ must be composite. In particular, $\sigma(n)$ must have a factor which is at least its square root.
Thus $\sigma(\sigma(n)) \geq \sigma(n) + \sqrt{\sigma(n)}.$

Now, assume $n$ is odd. Then by earlier remarks, we must have $\sigma(n)$ is a power of 2 times a perfect square, and $\sigma(n)$ is greater than 1. Thus, $\sigma(n)$ is composite. So whether $n$ is odd or even, we must have $\sigma(\sigma(n)) \geq \sigma(n) + \sqrt{\sigma(n)}.$

Thus, in all cases we must have $\sigma(\sigma(n)) \geq \sigma(n) + \sqrt{\sigma(n)}.$

One now need to only verify that if $\sigma(n) > 3n - \frac{\sqrt{12n+1}-1}{2}$, then $\sigma(n) + \sqrt{\sigma(n)} > 3n.$
\end{proof}

Our final restriction of this sort is the following. 
\begin{lemma} If $2|n$ and $5||n$ then $n \not \in Q$. 
\end{lemma}

Incorporating the results from this section into a straightforward search algorithm allows us to find that under 113,615,306,752  the only near superperfect numbers are 8, 21, 512, and Mersenne primes. The code  is written in Julia and available on Github \cite{SBeri}.

\section{Type II near superperfect numbers}

The definition of near superperfect is that $2n+d = \sigma(\sigma(n))$. However, there is reasonable objection to this being the ``correct'' definition, namely there is a type mismatch here. Since a number $n$ is superperfect if $2n = \sigma(\sigma(n))$, failing to be superperfect is a failure of the divisors of $\sigma(n)$ rather than those of of $n$. Thus, arguably a slightly more natural set to examine are numbers $n$ where $2n+d= \sigma(\sigma(n))$ and $d$ is a positive divisor of $\sigma(n)$. We will call such a number a \textit{Type II near superperfect number}, or just \textit{Type II} for short. In this section we will use the term Type I to refer to near superperfect numbers as originally defined by Kalita and Saikia. 

Type II numbers have some similarity to Type I. All Mersenne primes are also Type II for the trivial reason that for such numbers one always has $d=1$.

Unfortunately, many of the statements we can prove about Type I numbers do not apply to Type II. There are two particular failures of note. First, in the case of odd $n$, there is no obvious need for a Type II to have an odd value for $\sigma(\sigma(n))$ since $\sigma(n)$ may have even or odd divisors. Thus, we lack any obvious strong restriction on the factorization of $\sigma(n)$. Second, there does not seem to be any obvious way to force any upper bound on $\sigma(n)$ itself when $n$ is of Type II, since $d$ is no longer bounded from above by $n$. 

Consistent with these two failures to restrict Type II,  Type II numbers do seem to be more common. The first few Type II examples are:
$$3, 5, 7, 11, 19, 31, 41, 45, 103, 127, 271, 293, 463, 811, 1591, 1951, 8191, 93673.$$
This sequence does not appear in the OEIS.  It appears that all Type II numbers are odd. 

\begin{question} Are all type II numbers odd?
\end{question}

As with the search for near superperfect numbers, a search for Type II also naturally finds numbers where $d$ is negative, which we will call \textit{Type II deficient superperfect numbers}.

The first few Type II deficient superperfect numbers are
$$1, 13, 43, 109, 151, 181, 883, 1825, 2143, 2311, 20191, 30757, 268129, 321469. $$

This sequence is also not in the OEIS. Again, it seems that all of these are odd.

The sequence which is their union, $$1, 3, 5, 7, 11, 13, 19, 31, 41, 43, 45, 103, 127, 151, 181, 271 \cdots $$  is also not in the OEIS at this time.

We can at least prove that the natural density of the set of Type II near superperfect numbers and Type II deficient superperfect numbers is zero.

\begin{proposition} Let $C$ be the set of integers $n$ such that $2n+d = \sigma(\sigma(n)) $ for some $d$ which divides $\sigma(n)$. Then $C$ has natural density $0$. \label{stronger density zero claim} 
\end{proposition}
\begin{proof} The reasoning is similar to the proof of Proposition \ref{density zero}. The largest divisor of $\sigma(n)$ is $\sigma(n)$. Thus, if $\sigma(\sigma(n)) > \sigma(n) + 2n$, then $n$ cannot be an element of $C$. This will happen if $\sigma(\sigma(n)) > 3\sigma(n)$. The rest of the proof then proceeds identically.    
\end{proof}



Given the generalizations for $k$-near perfect, a similar generalization to $k$-near superperfect or $k$-near superperfect of Type II is straightforward. The same arguments as before show that for any fixed $k$, these sets are of natural density zero. We propose the following additional generalization.

We will say a number is $a$-$b$ \textit{hybrid superperfect} if there exist distinct but not necessarily positive divisors of $n$, $x_1,x_2, \cdots x_a$  and distinct but not necessarily positive  divisors $y_1, y_2 \cdots y_b$ of  $\sigma(n)$ such that 
\begin{equation} 2n + x_1 +x_2 +\cdots +x_a + +y_1 + y_2 \cdots + y_b = \sigma(\sigma(n)),    
\end{equation}
for some choice of signs. Given non-negative integers $a$ abd $b$ we will write $S_{a,b}$ as the set of $a$-$b$ hybrid numbers. Then using the similar logic as in the proof of Proposition \ref{stronger density zero claim} one can show that $S_{a,b}$ has density zero for any fixed $a$ and $b$. In contrast, it seems plausible that the union of all the $S_{a,b}$ has positive density.

\end{document}